\documentclass[journal,twoside,web]{ieeeconf}
\usepackage{cite}
\usepackage{graphicx}
\usepackage{textcomp}
\let\labelindent\relax
\usepackage{enumitem}
\usepackage{algorithm}
\usepackage{algpseudocode}
\usepackage{flushend}

\usepackage{amsmath, amssymb, amsthm, mathtools}
\usepackage{braket}      
\usepackage{physics}     
\usepackage{todonotes}
\usepackage{graphicx}    
\usepackage{subcaption}  
\usepackage{bm}          
\usepackage{xfrac}       
\usepackage{microtype}   
\usepackage{xcolor}

\usepackage[colorlinks=true, linkcolor=blue, citecolor=blue, urlcolor=blue]{hyperref}

\newcounter{prob}[section]
\newcounter{assmp}[section]

\theoremstyle{plain}
\newtheorem{theorem}{Theorem}[section]
\newtheorem{proposition}[theorem]{Proposition}

\theoremstyle{definition}
\newtheorem{definition}[theorem]{Definition}
\newtheorem{assumption}[assmp]{Assumption}
\newtheorem{problem}[prob]{Problem}

\theoremstyle{remark}
\newtheorem{remark}[theorem]{Remark}

\newcommand{\R}{\mathbb{R}}
\newcommand{\Rd}{\mathbb{R}^d}

\newcommand{\bx}{\mathbf{x}_N}

\newcommand{\bp}{\mathbf{p}_N}

\newcommand{\rev}[1]{\textcolor{black}{#1}}

\def\BibTeX{{\rm B\kern-.05em{\sc i\kern-.025em b}\kern-.08em
    T\kern-.1667em\lower.7ex\hbox{E}\kern-.125emX}}
\begin{document}
\title{Uniform Sampling from the Reachable Set Using \\ Optimal Transport}
\author{Karthik Elamvazhuthi and Sachin Shivakumar
\thanks{To be submitted for review.}
\thanks{This work was supported in part by the LDRD program of Los Alamos National Laboratory. }
\thanks{K. Elamvazhuthi \{karthikevaz@lanl.gov\} and S. Shivakumar \{sshivakumar@lanl.gov\} are with Applied Mathematics \& Plasma Physics Division, Los Alamos National Laboratory.}%
\thanks{S. Shivakumar is with Center for Nonlinear Studies, Los Alamos National Laboratory.}}

\maketitle

\begin{abstract}
\rev{
Estimating the reachable set of a dynamical system is a fundamental problem in control theory, particularly when control inputs are bounded. Direct simulation using randomly sampled admissible controls often leads to trajectories that cluster near attractors, resulting in poor coverage of the reachable set. To achieve a more uniform distribution of terminal states, we formulate the problem within an Optimal Transport (OT) framework. In this setting, the goal is to steer the system so that the final state distribution, determined by the chosen controls and initial conditions, matches a desired target distribution. Enforcing this condition exactly is not possible since the reachable set is not known. So we introduce an $L_2$-norm based regularization of the terminal distribution that relaxes the constraint while promoting uniform coverage. The resulting formulation can be approximated by a finite-dimensional, particle-based optimal control problem with kernel-coupled terminal cost. We show that this approach converges to the original formulation and demonstrate through numerical examples that it provides significantly more uniform reachable-set sampling than random control strategies.}


\end{abstract}


\section{Introduction}
\label{sec:introduction}
The reachable set associated with a control system describes the collection of all states that can be attained from a given initial condition under admissible control inputs over a finite time horizon. Given that it captures the range of system behavior it plays a central role in safety verification, constrained motion planning, and filtering \cite{bansal2017hamilton,jafarpour2023interval, coogan2015efficient}. In systems governed by nonlinear, underactuated, or input-constrained dynamics, the structure of the reachable set can be complex. \rev{It may lack convexity or admit no analytic description.} These features complicate both analysis and computation, particularly in contexts where the reachable set must be integrated, sampled, or used as a prior in estimation or control synthesis.

A number of works have addressed the issue of approximating the reachable set \cite{bansal2017hamilton,jafarpour2023interval, coogan2015efficient,thorpe2021learning,lew2021sampling,lew2025convex,devonport2021data}. Among these, sampling-based approaches have recently been explored for constructing approximations of the reachable set \cite{thorpe2021learning,lew2021sampling,lew2025convex,devonport2021data}. A common strategy involves selecting a collection of initial states and control inputs, evolving the system forward, and recording the resulting terminal states. This procedure implicitly induces a measure on the reachable set -- with an associated probabilistic guarantee on coverage of the reachable set --  given by the pushforward of the sampling distribution under the system’s endpoint map.

When probabilistic guarantees are provided in the context of sampling-based reachability computations, one must clarify the probability measure under consideration to understand the nature of such guarantees. In many settings, such guarantees are given with respect to the distribution induced by randomly chosen controls. Depending on the choice of control distribution, the resulting pushforward measure may concentrate on a lower-dimensional subset of the reachable set, particularly in systems with attractors or dominant modes. In such cases, regions of the reachable set can be underrepresented in the sample, reducing the effectiveness of the approximation for tasks that benefit from more uniform coverage, such as worst-case analysis or robust safety verification.

The goal of this paper is to address these challenges associated with sampling-based approximations of the reachable set by invoking ideas from optimal transport theory \cite{figalli2021invitation}. In particular, we draw upon developments in the optimal transport of nonlinear control systems \cite{agrachev2009optimal,rifford2014sub,elamvazhuthi2024benamou}. Beyond the original problem of transporting control systems between measures, optimal transport has already been found to be fruitful in related topic of uncertainty propagation of control systems \cite{halder2017gradient,aolaritei2025distributional}.

\rev{Before outlining the main contributions, we briefly note that our computational approach builds on the blob method \cite{carrillo2019blob,craig2025blob}, which approximates probability measures using interacting particles with smooth kernel influence. This particle-based, mesh-free formulation avoids grid-based PDE solvers for optimal transport and enables a scalable realization of the proposed reachability sampling framework.
}

\subsection*{Contribution}
First, note that the task of sampling from a set can be reformulated as transporting one probability measure to another supported on that set. Hence, optimal transport (OT) provides a natural conceptual framework for this problem.  Using this idea, we make the following contributions:

\begin{enumerate}
    \item \textbf{Transport formulation:} We pose the reachable set sampling problem as an Optimal Transport (OT) problem \rev{by searching for a joint measure on the product space of admissible initial conditions and controls such that the pushforward through the system’s endpoint map is the uniform distribution over the reachable set.}

    \item \textbf{Entropy-regularized relaxation:} \rev{While OT-based formulations have been used for sampling \cite{onken2021ot}, they typically assume that the target distribution is known. In contrast, the reachable set and its uniform measure are not known \emph{a priori}}. To address this, we introduce an entropy-type regularization using an $L_2$-based penalty that yields a tractable relaxation. To enable a particle approximation, this penalty is further regularized following the blob method for OT \cite{craig2025blob}. 

    \item \textbf{Convergence of the relaxed problem:} The entropy regularization, nonlocal blob-based smoothing, and finite-sample approximations define a hierarchy of optimization problems that successively approximate transport to a uniform measure. We establish that the minimizers of each level of this hierarchy converge (up to a subsequence) to minimizers of the next, ensuring consistency across relaxations.

    \item \textbf{Numerical Examples:} We demonstrate the performance of this method on nonlinear systems with strong attractors, where naive random sampling produces concentrated terminal distributions. The proposed method mitigates this attractor-bias and achieves a significantly more uniform coverage of the reachable set.
\end{enumerate}

\section*{Notation}
\rev{We introduce some notations and notions from optimal transport theory \cite{figalli2021invitation}.} \rev{For a Polish space, i.e., completely metrizable and seperable space $X$, we will denote the set of Borel probability measures on $X$ by $\mathcal{P}(X)$.} The set $\mathcal{P}_1(\Rd)$, will denote the subset of $\mathcal{P}(\Rd)$ with finite first moment -- i.e., $\mu \in \mathcal{P}_1(\Rd)$ implies $\int_{\Rd}|x|d\mu(x) < \infty$. For a Borel map $T : \mathbb{R}^d \to \mathbb{R}^d$ and $\mu \in \mathcal{P}(\mathbb{R}^d)$, $T_{\#} \mu$ denotes the push-forward of $\mu$ through $T$. The set of admissible controls, $\mathcal{U}$, is always equipped with the weak topology defined by the $L_2$ topology. We will say that $\mu << \nu$ if the Borel measure $\mu$ is absolutely 
continuous with respect to the Borel  measure $\nu$. The density of the measure $\mu$ with respect to $\nu$ will be denoted by $\frac{d\mu}{d\nu}$. If $\nu$ is the Lebesgue measure, denoted by $\lambda$, we will sometimes abuse the notation and identify $\mu$ with its density with respect to $\lambda$. Convergence in $\mathcal{P}(X)$ is the narrow convergence, denoted by $\mu_n\rightharpoonup\mu$, with respect to dual-pairing defined using the set of \rev{continuous}  bounded functions $C_b(X)$, i.e., a sequence $\{\mu_n\}\subset \mathcal{P}(X)$ narrow converges to $\mu\in \mathcal{P}(X)$ if and only if 
$\int_X f(x) d\mu_n(x) \to \int_X f(x) d\mu(x) \;\;\, \forall f\in C_b(X)$.
Convergence in $\mathcal{P}_1(\Rd)$ is the convergence in the Wasserstein-1 metric. Note the fact that if the sequences in $\mathcal{P}_1(\Rd)$ under consideration have a common compact support (usually, the reachable set for the purposes of this paper), then convergence in 
$\mathcal{P}_1(\Rd)$ is equivalent to narrow convergence. \rev{We denote by $L^2(0,T;\mathbb{R}^m)$ the space of square-integrable functions 
$u:[0,T]\to\mathbb{R}^m$ equipped with the norm 
$\|u\|_{L^2} = \left(\int_0^T \|u(t)\|^2\,dt\right)^{1/2}$.
Similarly, $\mathcal{C}([0,T];\mathbb{R}^d)$ denotes the space of continuous functions 
$x:[0,T]\to\mathbb{R}^d$ endowed with the supremum norm 
$\|x\|_{\infty} = \sup_{t\in[0,T]} \|x(t)\|$.
Throughout the paper, we use $\|\cdot\|$ to denote function norms, 
and $|\cdot|$ exclusively for finite dimensional 2 norm. We will typically work with the weak topology on bounded subsets of $L^2(0,T);\R^m$ which are metrizable and seperable. Hence, Polish.}

\section{Problem Formulation and Solution}\label{sec:problem}
First, let us formulate the reachability sampling problem for control-affine ODE systems in the form of an Optimal Transport (OT) problem. Consider the control system,
\begin{align}
\label{eq:ctrsys}
\dot{x}(t) = f(x(t), u(t)) : = f_0(x(t)) + \sum_{i=1}^m u_i (t)f_i(x(t)) ,
\end{align}
where \( x(t) \in \mathbb{R}^d \) and \( u(t) \in U \) denotes the control input. For some admissible control set \rev{$U \subset \R^m$}, the space of controls is defined as  
\[
\mathcal{U}=  \left\{ u \in L^2(0,T;\R^m)\;\mid\; u(t) \in U ~\text{for a.e} ~t \in (0,T) \right\}. \]
The set of admissible trajectories, $\Gamma \subset C([0,T];\mathbb{R}^d)$, is given by
\begin{equation}
\label{eq:trajectoryset}
\Gamma = \left\{ x \in C([0,T];\mathbb{R}^d) \;\middle|\; 
\begin{aligned}
&\exists\, u(\cdot) \in \mathcal{U} \text{ such that } \\
&x(t) \text{ solves } \eqref{eq:ctrsys} 
\end{aligned}
\right\}.
\end{equation}

Given a compact set of admissible initial conditions, \( \Omega \subseteq \mathbb{R}^d \), the \emph{reachable set} 
at time \( T \), denoted by \( \mathcal{R}^T_{\Omega} \subset \mathbb{R}^d \), 
is given by
\begin{equation}
\label{eq:reachableset}
\mathcal{R}^T_{\Omega}\! =\! \left\{ x_T \in \mathbb{R}^d \mid \exists \gamma \in \Gamma, \;\text{such that}\;\gamma(T) \!= \!x_T,\; \gamma(0)\! \in\! \Omega
\right\}.
\end{equation}
The main goal is then to uniformly sample from $\mathcal{R}^T_{\Omega}$.
\begin{problem}\label{prob:MainGoal}
Let $m_R$ be the Lebesgue measure restricted to the reachable set $\mathcal{R}^T_{\Omega}$. Then, generate uniformly distributed samples from $m_R$.
\end{problem}

To pose the above sampling problem as an OT problem, we first define $E : \Omega \times \mathcal{U} \rightarrow \Rd  $ to be the end point map given by
\begin{equation}
E(x_0,u) = \gamma(T)~~ \text{with}~~\gamma \in \Gamma ~\text{and}~\gamma(0) = x_0, 
\end{equation}
where $\gamma \in \Gamma$ is the trajectory corresponding to $u \in \mathcal{U}$. \rev{Here, we refer to $E:\Omega \times \mathcal{U} \to \mathbb{R}^d$ as the \emph{end-point map}, since it takes both the initial condition and control trajectory to the corresponding terminal state, although in the literature one typically fixes $x_0$.} \rev{To ensure regularity of this endpoint map, $E$, and well-posedness of the corresponding optimal transport problem we make the following assumptions that is standard in optimal control theory \cite{clarke2013functional}.}
\begin{assumption}[Smoothness and Linear Growth]
\label{asmp1:sublin}
Let $f_0, f_1, \dots, f_m : \mathbb{R}^d \to \mathbb{R}^d$ be the vector fields defining the control system. We assume:
\begin{enumerate}
    \item Each $f_i$ is \rev{continuously} differentiable, i.e., $f_i \in C^1(\mathbb{R}^d; \mathbb{R}^d)$ for all $i = 0, \dots, m$.
    \item There exists a constant $C > 0$ such that $\|f_i(x)\| \leq C(1 + \|x\|)$ for all $x \in \mathbb{R}^d$ and all $i = 0, \dots, m$.
\end{enumerate}
\end{assumption}

\begin{assumption}[Admissible Initial Conditions and Controls]
\label{asmp2:ctrl}
Given a dynamical system of the form \eqref{eq:ctrsys} whose vector fields satisfy Assumption \ref{asmp1:sublin}, we assume:
\begin{enumerate}
\item The set of admissible initial conditions $\Omega \subset \Rd$ is non-empty and compact. 
\item The set of controls, $U \subset \mathbb{R}^m$, is nonempty, compact, and convex.
\end{enumerate}
\end{assumption}
As a consequence of the two assumptions, the end-point map $E$ is continuous and the set $\mathcal{U}$ is compact in the weak topology in $L_2$. \rev{For instance, this is proved in~\cite[Thm.~23.11]{clarke2013functional}; see the Supplementary Material for a more detailed discussion.} Hence, the reachable set $\mathcal{R}^T_{\Omega}$ is compact.
Next, \rev{to address} Prob. \ref{prob:MainGoal}, we introduce an OT problem below.
\begin{problem}\label{prob:GoalOT} Given the Lebesgue measure of the reachable set, $c = \int_{\mathcal{R}^T_{\Omega}} dx$, solve the following transport problem \begin{align}
\label{eq:OTsamp}
\min_{
    \mathbb{P} \in \mathcal{P}(\Omega \times \mathcal{U})} \;\;
  \int_{\Omega \times \mathcal{U}}\int_{0}^T |u(t)|^2 dtd\mathbb{P}(x_0,u)
  \\
\text{s.t.,}~~E_{\#}\mathbb{P} = \frac{1}{c}\mathbf{1}_{\mathcal{R}^T_{\Omega}} = :m_R,\notag
\end{align}
\rev{where $\mathbf{1}_{\mathcal{R}^T_{\Omega}}$ is the indicator function on the reachable set.}
\end{problem}
Here, we abuse notation by identifying $m_{R}$ with the measure it defines on $\Rd$: $
m_R (A)=\frac{1}{c} \lambda (A \cap \mathcal{R}^T_{\Omega})$
for all Borel measurable sets $A \subseteq \Rd$.
\rev{The idea behind this new problem is that to uniformly sample from the reachable set, we solve a collection of optimal control problems that would be uniformly distributed on the reachable set.}
To ensure wellposedness of Prob. \ref{prob:GoalOT}, the following assumption is necessary.
\begin{assumption}
\label{asmp:fullleb}
The reachable set \( \mathcal{R}^T_{\Omega} \subset \mathbb{R}^d \) has positive and finite Lebesgue measure -- i.e., $0 < \lambda (\mathcal{R}^T_{\Omega}) < \infty$.
\end{assumption}
\rev{If the set of admissible conditions is a singleton, the above assumption holds true when the control system is locally controllable, or accessible. The assumption also holds true if the set of initial conditions has positive Lebesgue measure since the flow map for any fixed constant control is known to be Bi-Lipschitz.}

Problem \ref{prob:GoalOT} is distinct from the usual formulation of optimal transport with optimal control costs \cite{agrachev2009optimal}. Particularly, the projection of $\mathbb{P}$ on the initial conditions $\Omega$ is a variable. 
Using ideas from \cite{elamvazhuthi2024benamou} on the equivalence between the Kantorovich formulation and the Benamou-Brenier formulation, the formulations in \cite{agrachev2009optimal,rifford2014sub} can be shown to be equivalent to ours when $\Omega = x_0$, in which case we have posed an optimal transport problem of transferring from $\delta_{x_0}$ to $m_R$.

One can solve Problem \ref{prob:GoalOT} using the fluid dynamical formulation of optimal transport \cite{benamou2000computational}. However, it leads to a partial differential equation constrained optimization problem---typically solved by discretization. Instead, we use a finite sample approximation of this problem using the blob method for optimal transport introduced in \cite{craig2025blob}.

To develop a numerical method, first, we redefine Problem \ref{prob:GoalOT} as an unconstrained optimization problem by introducing a barrier functional, $\mathcal{F} \colon \mathcal{P}_1(\mathbb{R}^d) \ \longrightarrow\ [0,\infty]$, given
by\vspace{-3mm}
\begin{align*}
\mathcal{F}(\mu) :=
\begin{cases}
0, & \text{if } \mu = m_R, \\[0.3em]
\infty, & \text{otherwise}.
\end{cases}
\end{align*}

\vspace{-3mm}
Given this $\mathcal{F}$, the OT problem in \eqref{eq:OTsamp} is equivalent to
\begin{align}
\label{eq:OTsamp2}
\min_{
    \mathbb{P} \in \mathcal{P}(\Omega \times \mathcal{U})} 
 \mathcal{E}, \qquad \mathcal{E} := \mathcal{G}(\mathbb{P}) + \mathcal{F}(E_{\#} \mathbb{P}),
\tag{$\mathrm{ROT}$}
\end{align}
where $\mathcal{G}(\mathbb{P}) := \int_{\Omega \times \mathcal{U}}\int_0^T |u(t)|^2 \,dt\, d\mathbb{P}(x_0,u)$. However, to ensure numerical tractability, we first relax $\mathcal{F}$ to obtain
\begin{align*}
\min_{
    \mathbb{P} \in \mathcal{P}(\Omega \times \mathcal U)} 
  \int_{\Omega \times \mathcal{U}} \int_0^T |u(t)|^2 \, dt \, d\mathbb{P}(x_0, u)
  + \mathcal{F}_{\varepsilon}(E_{\#}\mathbb{P})
\end{align*}
where $\mathcal{F}_{\varepsilon}:\mathcal{P}(\mathbb{R}^d) \rightarrow \mathbb{R} \cup \left \{ \infty\right \}$ 
is some entropy penalization.

A common choice for entropy function in such situations is the functional $\rho \mapsto \varepsilon^{-1}\int \rho \log \rho$ (known as Shannon Entropy). However, to simplify the proof of convergence results we instead consider the $L_2$ energy function, \rev{for $\varepsilon >0$},
\begin{equation}\label{eq:L2entropy}
\mathcal{F}_{\varepsilon}(\mu) = 
\begin{cases}
\frac{1}{\varepsilon}\int_{\mathbb{R}^d} \big( \frac{d\mu}{d \lambda} \big )^2 dx, \; &\text{if}~~ \mu << \lambda,\\
\infty, &\text{otherwise},
\end{cases}
\end{equation}
where $\lambda$ is the Lebesgue measure on $\Rd$.
The justification for the penalization form \eqref{eq:L2entropy} is that it achieves an effect similar to Shannon entropy function. For a square integrable function on the compact set $\mathcal{R}_\Omega^T$, $\rho:\mathcal{R}_\Omega^T \rightarrow \R$, that integrates to $1$ on $\Omega$, the Jensen's inequality implies that $\int_{\mathcal{R}^T_\Omega} |\rho(x)|^2dx \geq (\int_{\mathcal{R}_\Omega^T} |\rho(x)|dx)^2$, where equality holds only if $\rho$ is constant. Therefore, the minimizer of $\mathcal{F}_{\varepsilon}(\mu)$ among all measures supported on the reachable set is the uniform probability measure and the minimum value is $\frac{1}{\lambda(\mathcal{R}_\Omega^T)}$. Thus, problem in \eqref{eq:OTsamp2} can be relaxed to 
\begin{align}
\label{eq:OTsamp2ep}
\min_{
    \mathbb{P} \in \mathcal{P}(\Omega \times \mathcal{U})}  
 \mathcal{E}_{\varepsilon},~~\mathcal{E}_{\varepsilon} : = \mathcal{G}( \mathbb{P})+ \mathcal{F}_{\varepsilon}(E_{\#}\mathbb{P}).
  \tag{$\mathrm{ROT}_{\varepsilon}$}
\end{align}

\section*{Particle Method}
Problem \eqref{eq:OTsamp2ep} is difficult to solve numerically since the decision variable $\mathbb{P}$ is infinite-dimensional. To obtain a numerically implementable algorithm, we can approximate the infinite-dimensional probability measure by a finite set of discrete measures (Dirac delta measures), which naturally leads to a particle-based optimization problem.
However, sums of dirac delta measures are not absolutely continuous with respect to $\lambda$ and $\mathcal{F}_\varepsilon$ is not well-defined for such measures. To overcome this issue, we smoothen the $\mathcal{F}_\varepsilon$ using a {\it symmetric positive mollifier} $k_\delta$ defined as follows.

\begin{definition}\label{def:mollifier}
A smooth, nonnegative, even function $k \in C_b(\mathbb{R}^d)$ is a \emph{symmetric positive mollifier} if it is compactly supported, $\int k(x) dx = 1$, and the limit, $\lim\limits_{\epsilon\to 0} \epsilon^{-d} k(y/\epsilon)$, is the Dirac delta function at $y$.    
\end{definition}

Given a mollifier $k$ with finite first-order moment, i.e., $k \in \mathcal{P}_1(\R^d)$, we define $k_\delta(y) = \delta^{-d} k(y/\delta)$---giving us a sequence of mollifiers with adjustable width parameter, $\delta$. For any probability measure $\nu\in \mathcal{P}(\R^d)$, the convolution of $\nu$ with $k_\delta$ is a bounded, continuous function, given by $ k_\delta*\nu(y) = \int_{\mathbb{R}^d} k_\delta (y-x) d \nu(x)$.
If we approximate $\mathcal{F}_\varepsilon$ by a sequence mollified functionals of the form
\begin{align} \label{Fepsdeldef}
\mathcal{F}_{\varepsilon,\delta}(\mu) : =  \frac{1}{\varepsilon}  \int_{\mathbb{R}^d} \left| k_\delta*\mu(y) \right|^2 dy,
\end{align}
we obtain an terminal entropy functional that is well-defined for Dirac delta measure.
This gives us the relaxed problem
\begin{align}
\label{eq:OTsamp2epdel}
\min_{
    \mathbb{P} \in \mathcal{P}(\Omega \times \mathcal{U})}  \mathcal{E}_{\varepsilon,\delta},~~~
  \mathcal{E}_{\varepsilon,\delta} : = \mathcal{G}( \mathbb{P})+ 
   \mathcal{F}_{\varepsilon,\delta}(E_{\#}\mathbb{P})
  \tag{$\mathrm{ROT}_{\varepsilon,\delta}$}
\end{align}
where, again, $\mathcal{G}( \mathbb{P}) =  \int_{\Omega\times \mathcal{U}} \int_{0}^T |u(t)|^2 \,dt\, d\mathbb{P}(x_0,u)$.
We note that, the square inside the integral of $\mathcal{F}_{\varepsilon,\delta}$ can be simplified as follows:
\begin{align*}
    &\varepsilon\mathcal{F}_{\varepsilon,\delta}(\mu)=\int_{\R^d} \left| k_\delta * \mu(y)\right|^2 dy\\
    &=\int_{\R^d} \int_{\R^d} \left(\int_{\R^d}k_{\delta}(y-x)k_{\delta}(y-z)dy\right)d\mu(z) d\mu(x)\\
    &:=\int_{\R^d}\left( \int_{\R^d} K_{\delta}(x-z) d\mu(z) \right)d\mu(x).
\end{align*}

Lastly, by approximating $\mu := E_{\#}\mathbb{P} \approx \frac{1}{N}\sum^N_{i=1} \delta_{x_i}$, we get the finite particle-based optimization problem 
\begin{align}
\label{eq:OTsamp2epdelN}
\min_{\substack{x_{i,0} \in \Omega \\ u_i \in \mathcal{U}}} 
& \Bigg\{
 \frac{1}{N} \sum_{i=1}^N \int_{0}^T |u_i(t)|^2\, dt \notag\\
& \quad + \frac{1}{N^2} 
  \frac{1}{\varepsilon} \sum_{i=1}^N \sum_{j=1}^N 
   K_\delta\!\big(x_i(T) - x_j(T)\big)
\Bigg\} \notag\\
\text{s.t.,} \quad 
& \dot{x}_i(t) = f\big(x_i(t),u_i(t)\big), \;\; x_i(0) = x_{i,0}.\tag{$\mathrm{ROT}_{\varepsilon,\delta,N}$}
\end{align}

\section{Main Results}
Having proposed successive relaxations of \eqref{eq:OTsamp2}, we must validate this hierarchy of approximations by establishing the convergence of minimizers of these approximations to the minimizer of the original problem. 
However, let us first establish some preliminary well-posedness results.

\begin{theorem}\label{thm:existence}
Given Assumptions~\ref{asmp1:sublin}, \ref{asmp2:ctrl} and \ref{asmp:fullleb}, and assuming a weak topology on $\mathcal{U}$, the optimization problem \eqref{eq:OTsamp2} is feasible. 
\rev{Moreover, an optimal solution exists for each of the optimization problems \eqref{eq:OTsamp2}, \eqref{eq:OTsamp2ep}, \eqref{eq:OTsamp2epdel}, and \eqref{eq:OTsamp2epdelN}.}
\end{theorem}
\begin{proof} By definition, analytic spaces are projection of subsets of spaces $X \times X $, where $X$ is Polish \cite[Section 1.3.4]{doberkat2007stochastic}.
Since \rev{$\mathcal{U}$ with the weak $L_2$-topology is metrizable \cite[Theorem~5.1]{conway2019course} and {\it Polish}, it is analytic. From the assumptions, the map $E: \Omega \times \mathcal{U} \mapsto \mathcal{R}^T_{\Omega}$ is continuous and surjective}. Therefore, the map
$E_{\#} : \mathcal{P}(\Omega \times \mathcal{U}) \rightarrow \mathcal{P}(\mathcal{R}^T_{\Omega})$ is also continuous and surjective \cite[Prop.~1.101]{doberkat2007stochastic}. \rev{Since $m_R\in \mathcal{P}(\mathcal{R}^T_{\Omega})$, from surjectivity of $E_{\#}$, there exists a $\mathbb{P} \in \mathcal{P}(\Omega \times \mathcal{U})$ such that $E_{\#}\mathbb{P} = m_R$. This proves feasibility of \eqref{eq:OTsamp2}.} 


\rev{Next, we prove existence of optimal solutions. For this, we will use the fact that optimal solution for a function exists if the function is lower semicontinuous on a compact space.}

\rev{The functionals $\mathcal{F}$, $\mathcal{F}_{\varepsilon}$ and $\mathcal{F}_{\varepsilon,\delta}$ are lower-semicontinuous on $\mathcal{P}_1(\Rd)$ by \cite[Proposition 3.9]{carrillo2019blob} and \cite[Lemma 3.1]{craig2025blob}.}
Since $E_{\#}$ continuous from $\mathcal{P}(\Omega \times \mathcal{U})$ to $\mathcal{P}_1(\Rd)$, the compositions
\[\mathbb{P} \mapsto \mathcal{F}(E_{\#}\mathbb{P}), \quad \mathbb{P} \mapsto \mathcal{F}_{\varepsilon}(E_{\#}\mathbb{P}), \quad \mathbb{P} \mapsto \mathcal{F}_{\varepsilon,\delta}(E_{\#}\mathbb{P}),\] inherit lower-semicontinuity of their corresponding functionals. Since the map $u \mapsto \|u\|^2_2$ is \rev{lower semicontinuous}, by \cite[Lemma 5.1.7]{ambrosio2005gradient}, $\mathcal{G}$ is \rev{lower semicontinuous}. Thus, $\mathcal{E}$, $\mathcal{E}_{\varepsilon}$ and $\mathcal{E}_{\varepsilon,\delta}$ are lower-semicontinuous.

Thus, using lower semicontinuity and compactness of  $\mathcal{P}(\Omega \times \mathcal{U})$ \rev{(follows from compactness of $\mathcal U$)}, we conclude that optimal solutions exist for \eqref{eq:OTsamp2}, \eqref{eq:OTsamp2ep}, and \eqref{eq:OTsamp2epdel}. Lastly, existence of optimal solutions for \eqref{eq:OTsamp2epdelN} has been established in classical control theory; See \cite[Theorem 23.11]{clarke2013functional}.
\end{proof}

\begin{remark}
When $\Omega = \lbrace x_0 \rbrace$ and the system is linear time invariant (LTI), the optimal solution of \eqref{eq:OTsamp2} is unique. This is proved by showing the time-reversed version of the problem is equivalent to a fixed end-point problem for an LTI system with a strictly convex objective.
\end{remark}

Next, we will relate minimizers of the problems 
\eqref{eq:OTsamp2}, \eqref{eq:OTsamp2ep}, \eqref{eq:OTsamp2epdel}, and \eqref{eq:OTsamp2epdelN} 
in the limits $\varepsilon \downarrow 0$, $\delta \to 0$, and $N \to \infty$ using $\Gamma$-convergence \cite{braides2002gamma}. For each case, we state a $\Gamma$–convergence result and a convergence of minimizers result. Since the proof closely resembles those in \cite{craig2025blob}, we will provide an outline, highlighting the differences. 

\begin{proposition}($\Gamma-$convergence as $\varepsilon \downarrow 0$)
\label{prop:gamma1}
Given Assumptions~\ref{asmp1:sublin}, \ref{asmp2:ctrl} and \ref{asmp:fullleb}, let $c = \int_{\mathcal{R}_\Omega^T} dx$. Then:
\begin{enumerate}[label=(\roman*)]
    \item  
    For every sequence $(\mathbb{P}_\varepsilon) \subset \mathcal{P}(\Omega \times \mathcal{U})$ converging
    to $\mathbb{P} \in \mathcal{P}(\Omega \times \mathcal{U})$ in the narrow topology, we have
    \[
        \mathcal{E}(\mathbb{P}) \ \le\ \liminf_{\varepsilon \downarrow 0} \ [\mathcal{E}_\varepsilon(\mathbb{P}_\varepsilon) - \frac{1}{c \varepsilon}].
    \]

    \item 
    For every $\mathbb{P} \in \mathcal{P}(\Omega \times \mathcal{U})$, we have
    \[
        \mathcal{E}(\mathbb{P}) \ \ge\ \limsup_{\varepsilon \downarrow 0} [ \mathcal{E}_\varepsilon(\mathbb{P}) - \frac{1}{c \varepsilon}].
    \]

    \item \label{it:liminfDel} For every sequence $(\mathbb{P}_\delta)\subset \mathcal{P}(\Omega \times \mathcal{U})$
    converging to $\mathbb{P}\in \mathcal{P}(\Omega \times \mathcal{U})$ in the narrow topology,
    \[
      \mathcal{E}_{\varepsilon}(\mathbb{P})
      \;\le\; \liminf_{\delta\to 0}\, \mathcal{E}_{\varepsilon,\delta}(\mathbb{P}_\delta).
    \]

    \item \label{it:limsupDel} For every $\mathbb{P}\in \mathcal{P}(\Omega \times \mathcal{U})$, we have
    \[
      \mathcal{E}_{\varepsilon}(\mathbb{P})
      \;\ge\; \limsup_{\delta\to 0}\,\mathcal{E}_{\varepsilon,\delta}(\mathbb{P}).
    \]

    \item For every sequence $(\mathbb{P}_N) \subset \mathcal{P}(\Omega \times \mathcal{U})$ converging
    to $\mathbb{P} \in \mathcal{P}(\Omega \times \mathcal{U})$ in the narrow topology, we have
    \[
        \mathcal{E}_{\varepsilon,\delta}(\mathbb{P})
        \ \le\ \liminf_{N \to \infty} \ \mathcal{E}_N(\mathbb{P}_N).
    \]

    \item For every $\mathbb{P} \in \rev{\mathcal{P}(\Omega\times \mathcal{U})}$, we have
    \[
        \mathcal{E}_{\varepsilon,\delta}(\mathbb{P})
        \ \ge\ \limsup_{N \to \infty} \ \mathcal{E}_N(\mathbb{P}).
    \]
\end{enumerate}
\end{proposition}

\begin{proof} 
\noindent\textbf{(i)}
Let $(\mathbb{P}_\varepsilon) \subset \mathcal{P}(\Omega \times \mathcal{U})$ be such that $\mathbb{P}_\varepsilon \rightharpoonup \mathbb{P}$. \rev{The claim is trivial for infinite \textit{liminf} and hence we consider finite \textit{liminf}. We pick a subsequence whose \emph{lim} is the \emph{liminf} since such a subsequence always exists}, i.e., pick $\mathbb{P}_\varepsilon$ such that
\begin{align*}
\liminf_{\varepsilon \downarrow 0} [\mathcal{E}_\varepsilon(\mathbb{P}_\varepsilon) -\frac{1}{c \varepsilon}] = \lim_{\varepsilon \downarrow 0} \mathcal{E}_\varepsilon(\mathbb{P}_\varepsilon) -\frac{1}{c \varepsilon}.
\end{align*}
Since convergent sequences are bounded, $\mathcal{F}_\varepsilon (E_{\#}\mathbb{P}_\varepsilon) -\frac{1}{c \varepsilon}$ is uniformly bounded over $\varepsilon>0$---implying $\lim_{\varepsilon \downarrow 0} \varepsilon \mathcal{F}_{\varepsilon}(E_{\#}\mathbb{P}_\varepsilon) =\frac{1}{c}$. 
Since $\mathcal{F}_\varepsilon$ is lower-semicontinous and $E_{\#}$ is continuous, the map $\mathbb{P} \mapsto \varepsilon\mathcal{F}_\varepsilon(E_{\#}\mathbb{P}) $ is also lower semicontinuous (l.s.c.). Thus, by definition of l.s.c.,
\[ \varepsilon \mathcal{F}_\varepsilon(E_{\#}\mathbb{P}) \leq \liminf_{\varepsilon \downarrow 0}   \varepsilon \mathcal{F}_\varepsilon(E_{\#}\mathbb{P}_\varepsilon) = \frac{1}{c}.  \]
By definition of $\mathcal{F}_\varepsilon$ and Jensen's inequality, we have that \rev{$\varepsilon \mathcal{F}_\varepsilon(E_{\#}\mathbb{P})\ge 1/c$} \rev{(see end of Sec.~II)}. Thus, $\mathcal{F}_\varepsilon(E_{\#}\mathbb{P}) = 1/c$, i.e., $E_\# \mathbb{P}$ is the uniform (Lebesgue) measure on the reachable set and $\mathcal{E}(\mathbb{P}) = \mathcal{G}(\mathbb{P})$. 
\rev{Since $\mathcal{G}$ is l.s.c. (see proof of Thm.\ref{thm:existence}), we have} 
\begin{align}
&\mathcal{E}(\mathbb{P}) 
= \mathcal{G}(\mathbb{P}) \leq \liminf_{\varepsilon \downarrow 0} 
    \mathcal{G}(\mathbb{P}_\varepsilon) \notag\\
&\leq \liminf_{\varepsilon \downarrow 0} 
    \big( \mathcal{G}(\mathbb{P}_\varepsilon) 
        + \mathcal{F}_\varepsilon(\mathbb{P}_\varepsilon) -\frac{1}{c \varepsilon} \big) = \liminf_{\varepsilon \downarrow 0} 
    [\mathcal{E}_\varepsilon(\mathbb{P}_\varepsilon) -\frac{1}{c \varepsilon}].\notag
\end{align}

\smallskip
\noindent\textbf{(ii)}
For any $\mathbb{P} \in \mathcal{P}(\Omega \times \mathcal{U})$, if $\mathcal{E}(\mathbb{P}) = \infty$, the claim is trivial. If $\mathcal{E}(\mathbb{P})$ is finite, then $\mathcal{F}(\mathbb{P}) = 0$ and $E_\# \mathbb{P}$ is uniform distribution on the reachable set. Thus, by definition of $\mathcal{F}_\varepsilon$, we have $\mathcal{F}_\varepsilon(\mathbb{P}) - \frac{1}{c \varepsilon} = 0,\; \forall \;\varepsilon >0$. {\color{blue} Adding $\mathcal{G}(\mathbb{P})$, we get  
\begin{equation}
\mathcal{E}_\varepsilon(\mathbb{P}) -\frac{1}{c \varepsilon} = \mathcal{G}(\mathbb{P}) + \mathcal{F}_\varepsilon(\mathbb{P}) -\frac{1}{c \varepsilon}
=  \mathcal{E}(\mathbb{P}),\; \forall~\varepsilon>0.\notag
\end{equation}}
Thus, the inequality in (ii) is trivially satisfied. The claims in (iii) and (v) can be proved by mirroring the arguments for (i) and likewise, (iv) and (vi) are proved mirroring (ii). 
\end{proof}

Given this result, we can prove convergence of minimizers of sequence of optimization problems.

\begin{proposition}[Convergence of minimizers]
\label{prop:OTsamp-conv}
Given Assumptions \ref{asmp1:sublin}, \ref{asmp2:ctrl}, \ref{asmp:fullleb}, and weak topology on $\mathcal{U}$, let $\mathbb{P}_\varepsilon \in \mathcal{P}(\Omega \times \mathcal{U})$ be a sequence of minimizers of~\eqref{eq:OTsamp2ep} 
, $\mathbb{P}_\delta \in \mathcal{P}(\Omega \times \mathcal{U})$ be a sequence of solutions to~\eqref{eq:OTsamp2epdel}, and $(\mathbf{x}_{N,0},\mathbf{u}_{N})\in \Omega^N\times\mathcal{U}^N$ be a sequence of solutions to \eqref{eq:OTsamp2epdelN} where 
    \[(\mathbf{x}_{N,0}, \mathbf{u}_N) :=\big ( (x_{N,0,1}, u_{N,1}) ,.... (x_{N,0,N}, u_{N,N})\big ).
    \] Then:
\begin{enumerate}[label=(\roman*)]
    \item there exists $\mathbb{P} \in \mathcal{P}(\Omega \times \mathcal{U})$ such that, upto a subsequence, $\mathbb{P}_\varepsilon \rightharpoonup\mathbb{P}$ and $\mathbb{P}$ is a minimizer of~\eqref{eq:OTsamp2}.
Futhermore, $E_{\#} \mathbb{P}_{\varepsilon}$ converges to the uniform Lebesgue measure on the reachable set.
 \item there exists $\mathbb{P} \in \mathcal{P}(\Omega \times \mathcal{U})$ such that, up to a subsequence, $\mathbb{P}_\delta \to \mathbb{P}$ in $\mathcal{P}(\Omega \times \mathcal{U})$, and $\mathbb{P}$ is \rev{a minimizer of} \eqref{eq:OTsamp2ep}.
    \item there exists $\mathbb{P} \in \mathcal{P}(\Omega\times \mathcal{U})$ such that, up to a subsequence, $\frac{1}{N} \sum_{i=1}^N \delta_{(x_{N,0,i},u_{N,i})} \to \mathbb{P}$ in $\mathcal{P}(\Omega \times \mathcal{U})$, and $\mathbb{P}$ is \rev{a minimizer of} \eqref{eq:OTsamp2epdel}.
\end{enumerate}

\end{proposition}
\begin{proof}
\textbf{Proof of (i).} Any minimizer of
\begin{align}\label{eq:augOTep}
\min_{
    \mathbb{P} \in \mathcal{P}(\Omega \times \mathcal{U})}  
 \mathcal{E}_{\varepsilon}(\mathbb{P})  -\frac{1}{c \varepsilon} 
\end{align}
is also a minimizer of~\eqref{eq:OTsamp2ep} \rev{since $1/c\varepsilon$ is a constant}. Therefore, it is sufficient to show minimizers of \eqref{eq:augOTep} converge to the minimizer of ~\eqref{eq:OTsamp2}. 




Let $\mathbb{P}_\varepsilon = \arg\min \left(\mathcal{E}_{\varepsilon}  -\frac{1}{c \varepsilon} \right)$. Since $\mathcal{U}$ is compact in weak topology, $\mathcal{P}(\Omega \times \mathcal{U})$ is compact in narrow topology. \rev{Since sequences in compact sets have convergent subsequences, there exists a converging subsequence of $\mathbb{P}_\varepsilon\rightharpoonup\mathbb{P}$, for some $\mathbb{P}\in \mathcal{P}(\Omega\times \mathcal{U})$.}
Then by Prop.~\ref{prop:gamma1}.i, we have
\[
\mathcal{E}(\mathbb{P})
\ \le\ \liminf_{\varepsilon \downarrow 0}\Bigl[\mathcal{E}_\varepsilon(\mathbb{P}_\varepsilon)-\frac{1}{c \varepsilon}\Bigr].
\]
Similarly by Prop. \ref{prop:gamma1}.ii, for any $\mathbb{Q} \in \mathcal{P}(\Omega \times \mathcal{U})$,  we have
\[
\limsup_{\varepsilon \downarrow 0}\Bigl[\mathcal{E}_\varepsilon(\mathbb{Q})-\frac{1}{c \varepsilon}\Bigr]
\ \le\ \mathcal{E}(\mathbb{Q}).
\]
\rev{Since $\mathbb{P}_\varepsilon$ is a minimizer, $\mathcal{E}_\varepsilon(\mathbb{P}_\varepsilon)-\frac{1}{c \varepsilon}\le \mathcal{E}_\varepsilon(\mathbb{Q})-\frac{1}{c \varepsilon}$ and}  
\[
\mathcal{E}(\mathbb{P})
 \!\le\! \liminf_{\varepsilon \downarrow 0}\Bigl[\mathcal{E}_\varepsilon(\mathbb{P}_\varepsilon)-\frac{1}{c \varepsilon}\Bigr]
 \!\le\! \limsup_{\varepsilon \downarrow 0}\Bigl[\mathcal{E}_\varepsilon(\mathbb{Q})-\frac{1}{c \varepsilon}\Bigr]
 \le \mathcal{E}(\mathbb{Q}),
\]
for all $\mathbb{Q}$. Hence $\mathbb{P}$ is a minimizer of \eqref{eq:OTsamp2}. 

\rev{For any subsequence $\mathbb{P}_{\varepsilon}\rightharpoonup \mathbb{P}$, we have $E_\#\mathbb{P}_\varepsilon\rightharpoonup E_\#\mathbb{P}$ since $E_\#$ is continuous, and $E_\#\mathbb{P} = m_R$, since $\mathbb{P}$ is a minimizer of \eqref{eq:OTsamp2}.}
Since $E_\#\mathbb{P}_\varepsilon$ lies in a compact set, and `every' convergent subsequence of $E_\#\mathbb{P}_\varepsilon$ converges to the same limit $m_R$, the sequence must converge to $m_R$. 
Similar arguments can be used to prove (ii) and (iii); See \cite{craig2025blob}. 

\end{proof}

\section{Numerical Examples}\label{sec:numerical}
In this section, we illustrate the effectiveness of the proposed algorithm by applying it for two nonlinear ODEs that have  strong attractors.

\subsection{Vander Pol Oscillator}
Consider the forced Vander Pol Oscillator, given by,
\begin{align}\label{eq:vanderpol}
    \dot{x}(t) &= y(t),\;\dot{y}(t) = \mu(1-x(t)^2)y(t)-x(t) + u(t).
\end{align}
Under zero control, the system exhibits a periodic orbit. In fact, the orbit is robustly stable in the sense that for a bounded control $|u(t)| \le u_{max}$ with small enough $u_{max}$, all trajectories converge to the same periodic orbit. Consequently, particle-based sampling methods, such as Monte Carlo, would lead to samples in the neighborhood of the orbit and does not represent the reachable subset well. As seen in Fig. \ref{fig:vanderpol:noentropy}, simulating $N=100$ particles via forward Euler time integration starting from bounded set $[-1,1]^2$ and applying a Brownian motion control projected onto the set $|u(t)|\le0.1$, lead to the periodic orbit by the final time, $t=15$.

\begin{figure}[!t]
    \centering
    \begin{subfigure}[b]{0.47\linewidth}
    \includegraphics[width=\textwidth]{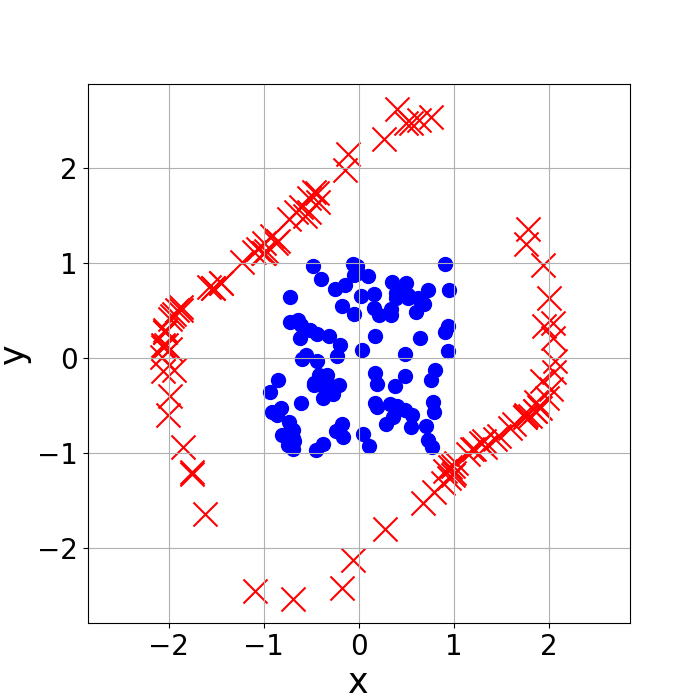}
    \caption{Random control inputs}
    \label{fig:vanderpol:noentropy}
    \end{subfigure}
    \begin{subfigure}[b]{0.47\linewidth}
    \includegraphics[width=\textwidth]{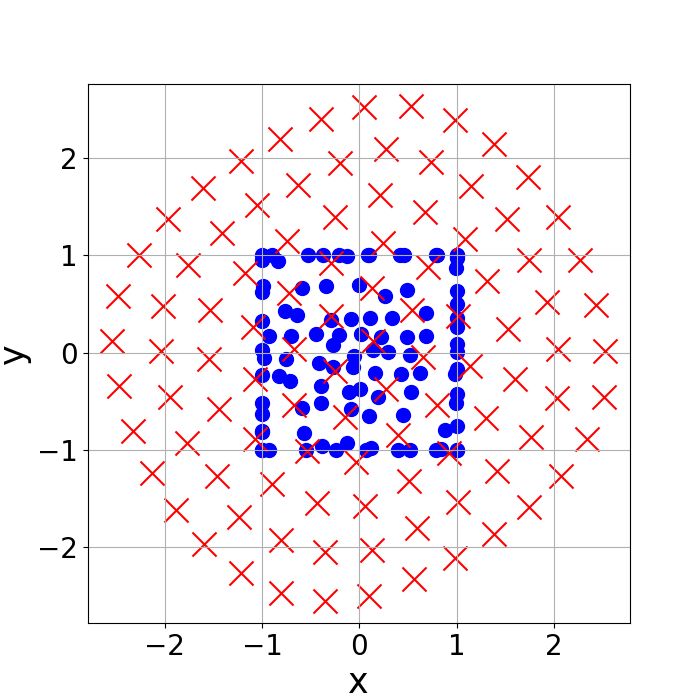}
    \caption{Controls from \eqref{eq:OTsamp2epdelN}}
    \label{fig:vanderpol:entropy}
    \end{subfigure}
    
    \caption{Final positions (at $t=15$, marked by `\textcolor{red}{$\times$}') of $N=100$ particles obeying Eq.~\eqref{eq:vanderpol} starting within $[-1,1]^2$ (marked by `\textcolor{blue}{$\bullet$}'). States evolve under two control laws: (a) projected Brownian control $u(t)\in[-0.1,0.1]$, and (b), control obtained by solving the optimization problem in \eqref{eq:OTsamp2epdelN}.}
    \label{fig:vanderpol}
\end{figure}

In contrast, there exist controls that uniformly sample from the reachable set; See Fig. \ref{fig:vanderpol:entropy}. Here, the simulation parameters are identical to the case in Fig. \ref{fig:vanderpol:noentropy}, however, applied controls are obtained by using Pontryagin's Maximum principle based gradient descent to solve \eqref{eq:OTsamp2epdelN}. 
\subsection{Three-arm Planar Robot}

We next consider the reachability of a planar rigid robot with three actuated revolute joints with joint lengths $(1.0,\,0.3,\,0.4)$, and masses $(1.0,\,0.3,\,0.2)$, respectively. Let $q(t)=(q_1,q_2,q_3)^\top$ be the joint angles and $\omega(t)=\dot q(t)$ the joint velocities. The dynamics of this six-dimensional system are 
\begin{align}\label{eq:3arm}
\dot q(t) &= \omega(t),\;
M \dot \omega(t) = -D\,\omega(t) - gG\!\big(q(t)\big) + u(t),
\end{align}
where $u(t)\in[-0.2, 0.2]^3$, $D=0.05$ is viscous damping, $M=\mathrm{diag}(J_1,J_2,J_3)$ is the inertia matrix, and $G(q)$ are gravity torques. Denoting link center-of-mass offsets $\bar l_i=\tfrac{1}{2}l_i$ and $I_i=\tfrac{1}{12}m_i l_i^2$ as the link inertias, we get \vspace{-1mm}
\begin{align*}
J_1 &= I_1 +m_1 \bar l_1^2 + m_2l_1^2 + m_3l_1^2 +J_2,\\
J_2 &= I_2 +m_2 \bar l_2^2 + m_3l_2^2 +J_3,\quad J_3 = I_3 + m_3 \bar l_3^2,\\
G_1(q) &= m_1 \bar l_1 \cos q_1 + m_2 l_1\cos q_1 + m_3 l_1\cos q_1 +G_2(q),\\
G_2(q) &= m_2 \bar l_2 \cos (q_1+q_2) + m_3 l_2\cos (q_1+q_2) + G_3(q),\\
G_3(q) &= m_3 \bar l_3 \cos (q_1+q_2+q_3).
\end{align*}

\vspace{-1mm}
As seen in Fig.~\ref{fig:3arm:noentropy}, projected Brownian torques do not cover the workspace of the end-effector, however, \eqref{eq:OTsamp2epdelN} produces controls that substantially improve reachable set approximation; see Fig.~\ref{fig:3arm:entropy}.

\begin{figure}[!t]
    \centering
    \begin{subfigure}[b]{0.47\linewidth}
     \includegraphics[width=\linewidth]{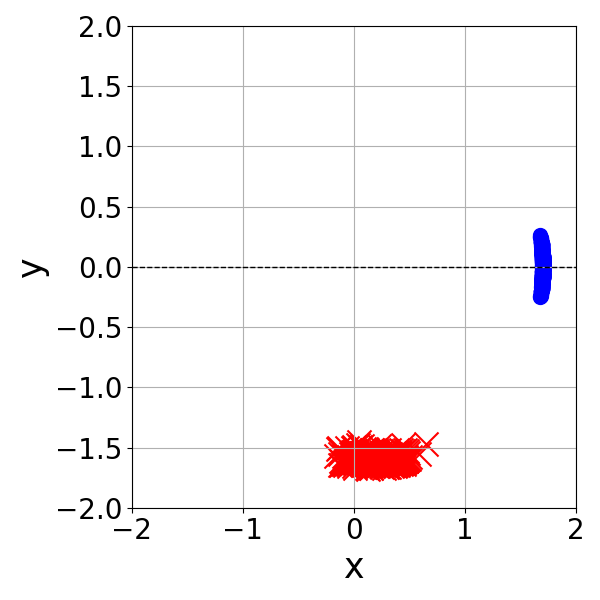}
    \caption{Random control inputs}
    \label{fig:3arm:noentropy}
    \end{subfigure}
    \begin{subfigure}[b]{0.47\linewidth}
     \includegraphics[width=\textwidth]{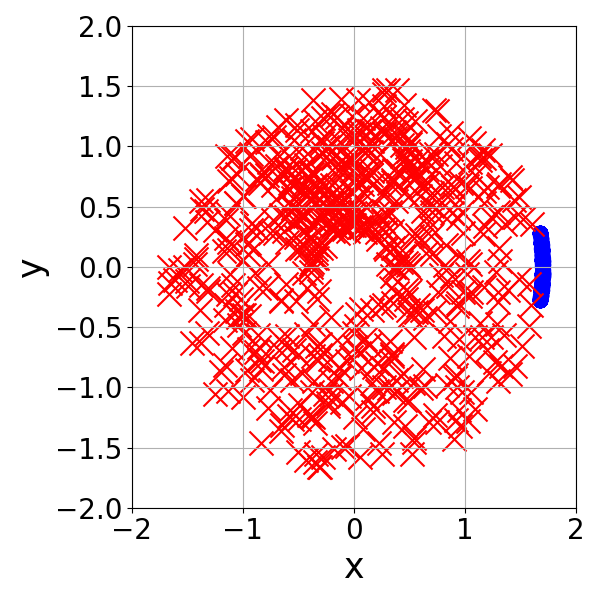}
     \caption{Controls from \eqref{eq:OTsamp2epdelN}}
    \label{fig:3arm:entropy}    
    \end{subfigure}
    \caption{Final positions (at $t=20$, marked by `\textcolor{red}{$\times$}') of $N=600$ particles obeying Eq.~\eqref{eq:3arm} starting with joint angles and velocities within $[-0.1,0.1]$ (marked by `\textcolor{blue}{$\bullet$}'). States evolve under two control laws: (a) projected Brownian control $u(t)\in[-0.2,0.2]$, and (b), control obtained by solving the optimization problem in \eqref{eq:OTsamp2epdelN}.}
    \label{fig:pendulum}
\end{figure}

\section{Conclusion}
In conclusion, we have provided a optimal transport formulation of sampling from the reachable set. Using a series of regularizations we developed a finite dimensional optimal control problem that approximates the optimal transport formulation. Moreover, leveraging tools from Gamma convergence, we proved that minimizer converge up to (subsequences) to the limit problem. We validated the approach through  numerical experiments on the Van der Pol oscillator and a model of three-arm robot. 

\bibliographystyle{ieeetr}
\bibliography{ref}

\newpage

\appendix

\subsection{More Examples}

In this section, we apply the results presented in this paper and solve the uniform sampling of reachable set for additional examples -- specifically, the Duffing Oscillator.

\subsubsection{Duffing Oscillator}
Consider the (bistable) Duffing oscillator with state $(x,y)$, given by
\begin{equation}
\label{eq:duffing}
\dot x = y, \; \dot y = -\delta\, y + \alpha\, x - \beta\, x^{3} + u,
\end{equation}
where $\delta=0.3$ is the linear damping, and $\alpha=\beta=1$. $[0,0]$ is an unstable equilibrium and any non-zero perturbative force, $u(t)$, generates chaotic trajectories converging to $(\pm\sqrt{\alpha/\beta},0)$.

Similar to the examples in Sec. \ref{sec:numerical}, we simulate $N=100$ particles using forward Euler time integration while applying a stochastic control $u(t)\in[-1,1]$ at each time-step. See Fig. \ref{fig:duffing:noentropy} shows that all particles starting from origin, after time $t=5$ are restricted to a small subspace of the reachable set due to the presence of strong attractors at $(\pm\sqrt{\alpha/\beta},0)$.

\begin{figure}[!ht]
    \centering
    \includegraphics[width=0.8\linewidth]{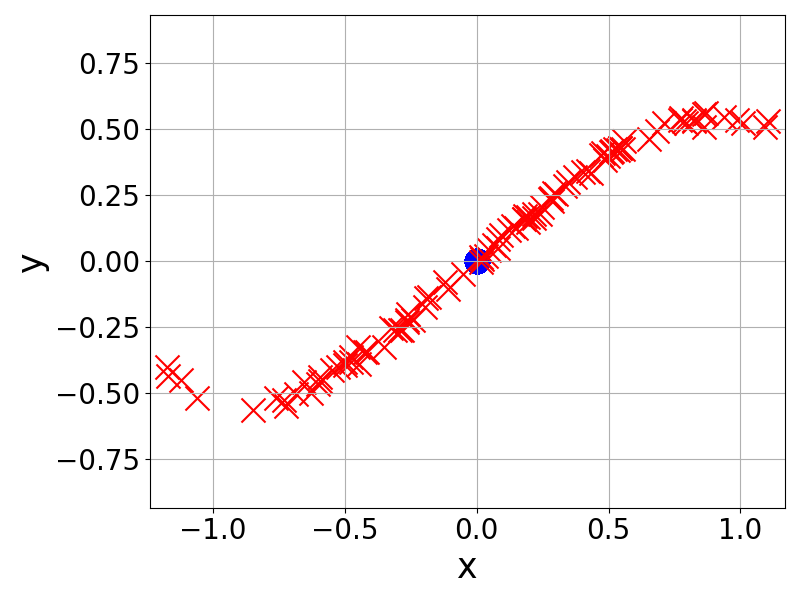}
    \caption{Final positions (at $t=5$, marked by `\textcolor{red}{$\times$}') of $N=100$ particles obeying Eq.~\eqref{eq:duffing} starting at $[0,0]$ (marked by `\textcolor{blue}{$\bullet$}'). States evolve under uniformly distributed stochastic control $u(t)\in[-1,1]$.}
    \label{fig:duffing:noentropy}
\end{figure}

However, applying the controls obtained by  solving \eqref{eq:OTsamp2epdelN}, we obtain a better representation of the reachable set, which is uniformly sampled; See Fig. \ref{fig:duffing:entropy}. 

\begin{figure}
    \centering
     \includegraphics[width=0.8\linewidth]{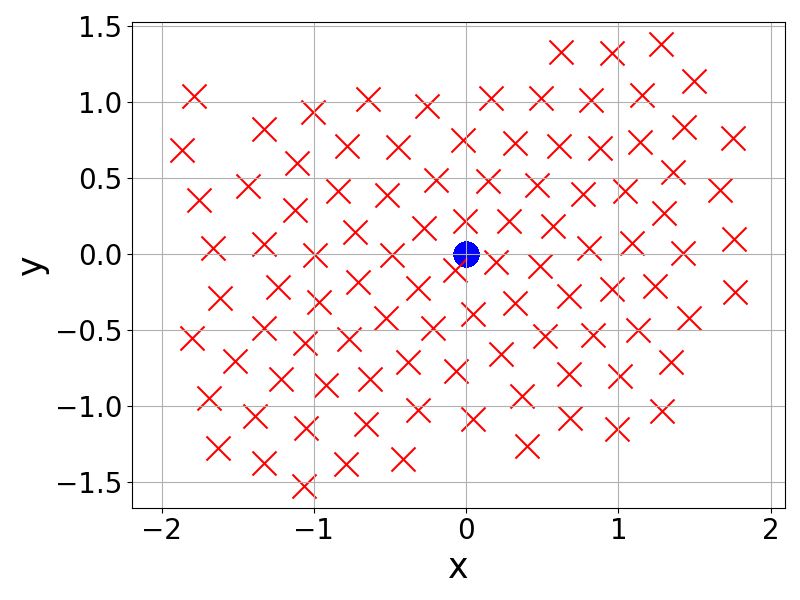}   
    \caption{Final positions (at $t=5$, marked by `\textcolor{red}{$\times$}') of $N=100$ particles obeying Eq.~\eqref{eq:duffing} starting at $[0,0]$ (marked by `\textcolor{blue}{$\bullet$}'). States evolve under the control obtained by solving the optimization problem in \eqref{eq:OTsamp2epdelN}.}
    \label{fig:duffing:entropy}
\end{figure}

\subsubsection{Single-arm Robot}
Next, we consider the reachability of a rigid-robot arm actuated by a torque motor. In practice, this system can be modeled by a forced nonlinear Pendulum with damping, whose dynamics are given by
\begin{align}\label{eq:nonlinearpendulum}
    \dot{x}(t) &= y(t),\; \dot{y}(t) = -\frac{g}{l}\sin(x(t)) - \beta y(t) + u(t),
\end{align}
where $x$ is the angular position of the robotic arm, $y(t)$ is the angular velocity, $u(t)$ is the torque applied, $g=9.81 m/s$ is acceleration due to gravity, $l=1$ is length of the arm, and $\beta =0.1$ is damping coefficient. Assuming the motor has torque constraints, $u(t)\in [-1,1]$, we can run $N=100$ Monte Carlo simulations of the robot, starting with varied position and angular velocities $[x_i(0), y_i(0)] \in [-1,1]^2$. Despite simulating for $t=15$, one cannot accurately predict all reachable configurations of the robot under this bounded input, since zero angular displacement, $x(t)=0$, is a robustly stable attractor; See Fig. \ref{fig:pendulum:noentropy}. However, by solving \eqref{eq:OTsamp2epdelN}, one can find control inputs that show a better approximation of all reachable configurations; See Fig. \ref{fig:pendulum:entropy}.


\begin{figure}[!t]
    \centering
     \includegraphics[width=0.8\linewidth]{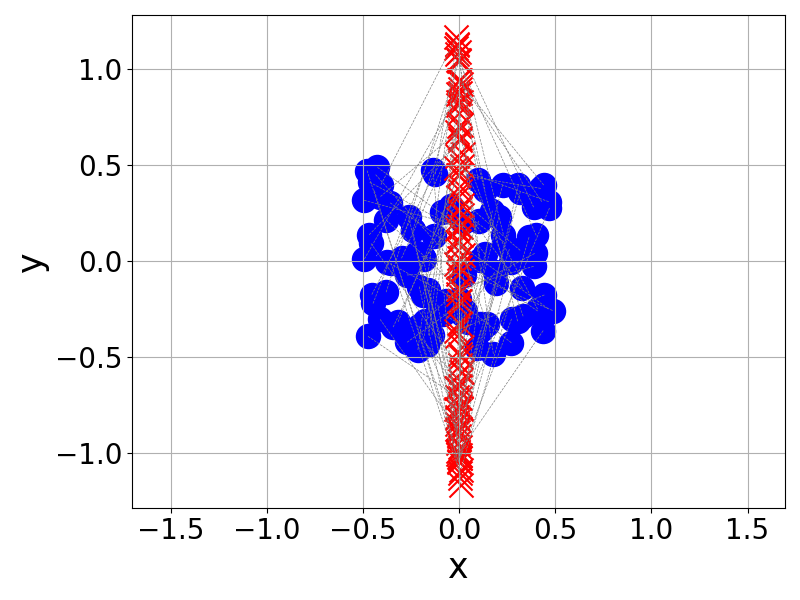}
    \caption{Final positions (at $t=15$, marked by `\textcolor{red}{$\times$}') of $N=100$ particles obeying Eq.~\eqref{eq:nonlinearpendulum} starting within $[-1,1]^2$ (marked by `\textcolor{blue}{$\bullet$}'). States evolve under uniformly distributed stochastic control $u(t)\in[-1,1]$.}
        \label{fig:pendulum:noentropy}
\end{figure}

\begin{figure}[!t]
    \centering
     \includegraphics[width=0.8\linewidth]{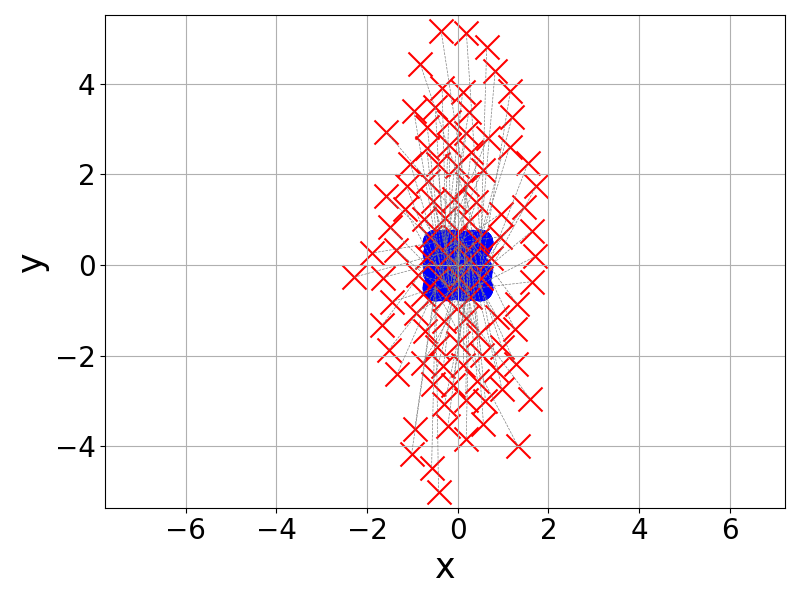}    
    \caption{Final positions (at $t=15$, marked by `\textcolor{red}{$\times$}') of $N=100$ particles obeying Eq.~\eqref{eq:nonlinearpendulum} starting within $[-1,1]^2$ (marked by `\textcolor{blue}{$\bullet$}'). States evolve under control obtained by solving the optimization problem in \eqref{eq:OTsamp2epdelN}.}
    \label{fig:pendulum:entropy}
\end{figure}

\subsection{Numerical Algorithm}
In this section we describe the numerical method used to solve the 
problem~\eqref{eq:OTsamp2epdelN}. For notational convenience we introduce the
vectorized form of the optimal control problem:
\begin{align}
&\inf_{\substack{\mathbf{x}_0 \in \R^{Nd}\\ \mathbf{u}_N \in \mathcal{U}^N}}
   \frac{1}{N}\int_0^T |\mathbf{u}_N(t)|^2\,dt \;+\; \Psi\big(\mathbf{x}_N(T)\big),\\
&\dot{\mathbf{x}}_N = F_N(\mathbf{x}_N,\mathbf{u}_N),\quad \mathbf{x}_N(0) = \mathbf{x}_{0,N}.\notag
\end{align}
Here the stacked state and initial condition vectors are
$\mathbf{x}_N(t) = (x_1(t),\ldots,x_N(t))^\top \in \R^{Nd}$ and
$\mathbf{x}_{0,N} = (x_{1,0},\ldots,x_{N,0})^\top \in \Omega^N \subset \R^{Nd}$.
The stacked control vector is
$\mathbf{u}_N(t) = (u_1(t),\ldots,u_N(t))^\top \in \R^{Nm}$,
and the dynamics are given by
$F_N(\mathbf{x}_N,\mathbf{u}_N) = (f(x_1,u_1),\ldots,f(x_N,u_N))^\top \in \R^{Nd}$.

To compute a local minimizer we employ a gradient descent scheme. The gradient
is characterized by the adjoint equation, as in Pontryagin’s maximum principle:
\begin{align}
\label{eq:adjN}
\dot{\mathbf{p}}_N &= -A_N(\mathbf{x}_N,\mathbf{u}_N)^\top \bp,\;
\bp(T) = \nabla_{\bx}\Psi\big(\bx(T)\big),
\end{align}
where $A_N(\mathbf{x}_N,\mathbf{u}_N)=\mathrm{diag}\big(A_1(x_1,u_1),\ldots,A_N(x_N,u_N)\big)$ and $A_i(x_i,u_i) = \nabla_{x_i} f(x_i,u_i)$. Applying the chain rule in the setting of functionals on Banach spaces, the
gradients of $J$ with respect to the initial state and the control are
\begin{align}
\label{eq:gradsN}
\nabla_{\mathbf{x}_{0,N}} J(\mathbf{x}_{0,N},\mathbf{u}_N) &= \mathbf{p}_N(0),\\[1ex]
\nabla_{\mathbf{u}_N} J(\mathbf{x}_{0,N},\mathbf{u}_N)(t) &= B(t)^\top \mathbf{p}_N(t) + \tfrac{2}{N}\,\mathbf{u}_N(t),\notag
\end{align}
with $B(t)=\mathrm{diag}\big(B_1(t),\ldots,B_N(t)\big)$ where $B_i(t) = \nabla_{u_i} F_N(x_i,u_i)$.

\begin{algorithm}[H]
\caption{Projected Gradient Descent for Problem \eqref{eq:OTsamp2epdelN}}
\begin{algorithmic}[1]
\State Initialize $\mathbf{x}_{0,N}^{(0)} \in \Omega^N$ and $\mathbf{u}_N^{(0)} \in \mathcal{U}^N$.
\For{$k = 0,1,2,\dots$ until convergence}
    \State Forward simulate the state dynamics \eqref{eq:OTsamp2epdelN} with $\mathbf{x}_{0,N}^{(k)}$ and $\mathbf{u}_N^{(k)}$.
    \State Backward simulate the adjoint system \eqref{eq:adjN} with terminal condition.
    \State Compute gradients using \eqref{eq:gradsN}.
    \State Update initial state (projected step):
    \[
      \mathbf{x}_{0,N}^{(k+1)} \gets 
      \Pi_{\Omega^N}\!\Big(\mathbf{x}_{0,N}^{(k)} 
      - \eta_k \,\nabla_{\mathbf{x}_{0,N}} J\big(\mathbf{x}_{0,N}^{(k)},\mathbf{u}_N^{(k)}\big)\Big).
    \]
    \State Update control (projected step, pointwise in $t$):
    \[
      \mathbf{u}_N^{(k+1)}(t) \gets 
      \Pi_{\mathcal{U}^N}\!\Big(\mathbf{u}_N^{(k)}(t) 
      - \eta_k \,\nabla_{\mathbf{u}_N} J\big(\mathbf{x}_{0,N}^{(k)},\mathbf{u}_N^{(k)}\big)(t)\Big).
    \]
\EndFor
\end{algorithmic}
\end{algorithm}

\subsection{Connection to Classical Optimal Transport}
The solution of the optimization problem~\ref{eq:OTsamp} can also be formulated as a classical optimal transport problem using Kantorovich formulation. To see this,
let \( c(x_0, x_T) \) denote the cost function associated with the optimal control problem
\begin{align*}
c(x_0, x_T) 
&= \inf_{u \in \mathcal{U}} &&\int_0^T |u(t)|^2 \, dt, \\
&\text{s.t. } &&Eq.~\eqref{eq:ctrsys}, ~x(0) = x_0,~ x(T) = x_T.
\end{align*}
Next, we define \( F : \mathbb{R}^d \times \mathcal{U} \rightarrow \mathbb{R}^d \times \mathbb{R}^d \) to be the map 
\begin{equation*}
F(x_0, u) = \big(x_0, E(x_0, u)\big),
\end{equation*}
where $E$ is the endpoint map for Eq. \eqref{eq:ctrsys}.
Given $\mathbb{P}\in \mathcal{P}(\R^d\times \mathcal{U})$, if \( \eta = F_{\#} \mathbb{P} \), then
\begin{align*}
&\int_{\mathbb{R}^d \times \mathbb{R}^d}\!\! c(x_0, x_T) \, d\eta(x_0, x_T) 
=\! \int_{\mathbb{R}^d \times \mathcal{U}} \!\!c\big(F(x_0, u)\big) \, d\mathbb{P}(x_0, u) \\
&\quad =\! \int_{\mathbb{R}^d \times \mathcal{U}}\!\! c\big(x_0, E(x_0, u)\big) \, d\mathbb{P}(x_0, u) \\
&\quad = \!\int_{\mathbb{R}^d \times \mathcal{U}} |u|^2 \, d\mathbb{P}(x_0, u).
\end{align*}
Therefore, our transport problem in \eqref{eq:OTsamp2} can be written in the more familiar Kantorovich form as
\begin{align*}
\inf_{\eta \in \mathcal{P}(\mathbb{R}^d\times \Rd)} 
    \int_{\mathbb{R}^d \times \mathbb{R}^d} c(x_0, x_T) \, d\eta(x_0, x_T)
    \\
\text{subject to } (\pi_2)_{\#}\eta = m_R,
\end{align*}
where \( \pi_2(x_0, x_T) = x_T \) is the projection onto the second coordinate.

The above derivation shows that our problem achieves a higher infimum than the Kantorovich version of the problem. However, one can also prove that the infima are equal by showing that for a measure $\eta$, one can construct a suitable $\mathbb{P}$. Thus, demonstrating equivalence. This is easy to see, for instance, if for each $x_0$ and $x_T$ there is unique optimal control map $F^{\rm inv}(x_0,x_T) = (x_0,u)$.

Note that this represents a relaxed version of the classical optimal transport problem, since the first marginal of the transport plan is allowed to vary, whereas in the classical formulation it is typically fixed. 

\subsection{Proof of Continuity of the End-point Map}
In this section, we prove the continuity of the end-point map.
\begin{theorem}
Given Assumption \ref{asmp1:sublin}, the end-point map $E:\Rd \times \mathcal{U} \rightarrow \Rd$ is continuous on $\Rd \times \mathcal{U}$ for the standard Euclidean topology on $\Rd$ and the weak topology on $\mathcal{U}$.
\end{theorem}
\begin{proof}
Let $(x_0^n, u^n)_{n=1}^\infty$ be a sequence converging to some $(x_0^*, u^*) \in \mathbb{R}^d \times \mathcal{U}$ (with respect to the product topology induced by the Euclidean topology on $\mathbb{R}^d$ and the weak topology on $\mathcal{U}$). This also implies that both $(x_0^n)$ and $(u^n)$ are uniformly bounded when $\mathcal{U}$ is endowed with the $L_2$-norm topology. Let $(x^n)_{n=1}^{\infty} \subset C([0,T]; \mathbb{R}^d)$ denote the corresponding solutions of \eqref{eq:ctrsys}. By standard compactness arguments (see, for instance, the proof of Theorem~23.11 in \cite{clarke2013functional}), one can extract a subsequence of $(x^n)_{n=1}^\infty$ that converges to a limit trajectory $x^*$, which is itself a solution of \eqref{eq:ctrsys}. The main difference from the setting in \cite{clarke2013functional} is that we also vary the initial condition $x_0^n$; however, this does not affect the argument, since uniform bounds on the integral of $|\dot{x}^n(t)|$ are ensured by the linear growth assumption. Hence, by Arzelà--Ascoli’s theorem and the equicontinuity of the sequence $(x^n)$, we obtain uniform convergence in $C([0,T];\Rd)$. The limit $x^*$ is unique, regardless of the subsequence chosen, because \eqref{eq:ctrsys} admits a unique solution corresponding to $(x_0^*, u^*)$. This establishes the continuity of the end-point map.

\end{proof}

\subsection{A Qualitative Comparison}
In this section, we utilize an alternative sampling-based algorithm to find the reachable sets for the examples presented in Sec.~\ref{sec:numerical}. More specifically, we implemented the sampling-based algorithm presented in a recent work,  \cite{lew2025convex}, for reachability analysis of nonlinear systems with bounded controls and initial conditions. By comparing the Figs. \ref{fig:lew:3arm}, \ref{fig:lew:vanderpol} with those in Sec.~\ref{sec:numerical}, on can make an empirical observation that the proposed algorithm performs equally well, if not better, than existing state-of-the-art for finding reachable sets. In particular, the reachable set for the 3-armed robot is a torus, which is non convex. Using convex hulls to approximate this reachable set is diffcult and the results are often conservative. However, the proposed particle-based OT formulation does not suffer from this issue.
\begin{figure}[!b]
    \centering
     \includegraphics[width=0.8\linewidth]{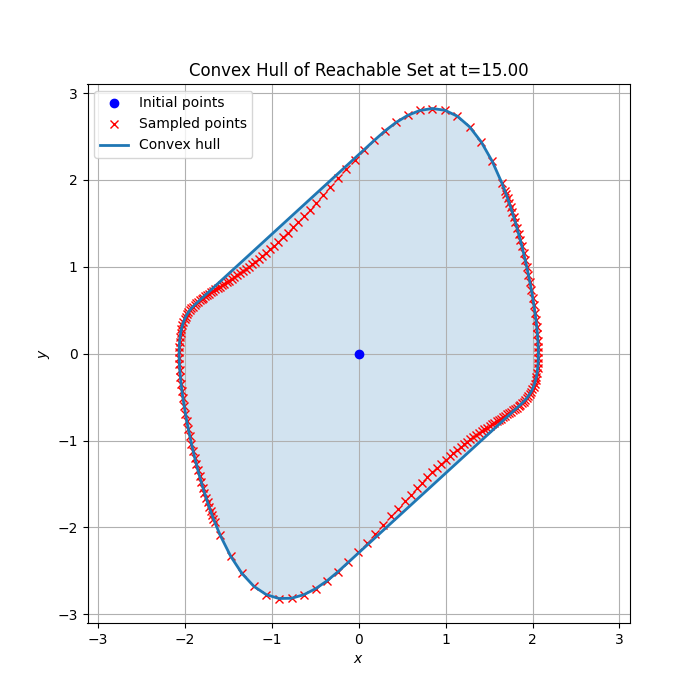}
    \caption{Plotting convex hulls of reachable set for Vanderpol Oscillator at $t=15$ under bounded controls $u(t)\in [-0.1,0.1]$ obtained using the algorithm in \cite{lew2025convex}.}
        \label{fig:lew:vanderpol}    
\end{figure}
\begin{figure}[!b]
    \centering
     \includegraphics[width=0.8\linewidth]{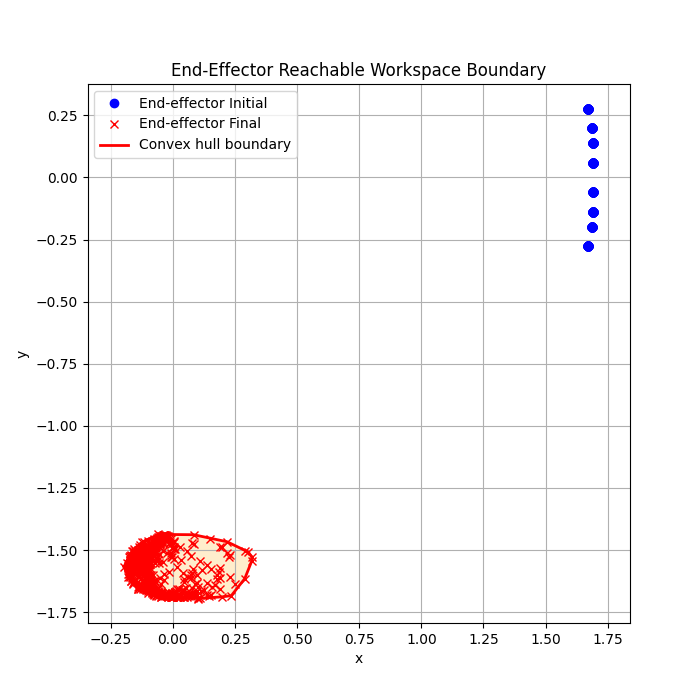}
        \caption{Plotting convex hulls of reachable set for 3-arm Planar robot at $t=20$ under bounded controls $u(t)\in [-0.2,0.2]$ obtained using the algorithm in \cite{lew2025convex}.}
    \label{fig:lew:3arm}    
\end{figure}

\end{document}